\documentclass[11pt,leqno]{amsart}
\usepackage[utf8]{inputenc}
\usepackage[utf8]{inputenc}
\usepackage{amsfonts,amssymb}
\usepackage{graphicx,tikz}
\usepackage{tikz-cd} 
\usepackage{float}
\usepackage{amsmath, amsthm, amsfonts}
\usepackage{geometry}
\usepackage{hyperref}
\usepackage{enumitem}  
\usepackage[capitalise]{cleveref}
\usepackage{comment}
\usepackage{enumitem}
\usepackage{amsrefs}
\usepackage{nicefrac}
\usepackage[official]{eurosym}
\usepackage{graphicx}
\usepackage{nomencl}
\usepackage{amsfonts}
\usepackage{hyperref}
\usepackage{amssymb}
\usepackage{fancyhdr}
\usepackage{amscd}
\usepackage[english]{babel}

\newtheorem{theorem}{Theorem}[section]
\newtheorem{thm}[theorem]{Theorem}
\newtheorem{lem}[theorem]{Lemma}
\newtheorem{pro}[theorem]{Proposition}

\title[Neighborhoods in enhanced power graphs]{On neighborhoods in the enhanced power graph\\ associated with a finite group}
\author[M. L. Lewis]{Mark L. Lewis}
\address{Department of Mathematical Sciences,
Kent State University, Kent, OH  44242, U.S.A.}
\email{(Lewis) lewis@math.kent.edu}
\author[C. Monetta]{Carmine Monetta}
\address{Dipartimento di Matematica, Universit\`a di Salerno, Fisciano 84084 (SA), Italy}
\email{(Monetta) cmonetta@unisa.it}


\begin{document}

\begin{abstract}
This article investigates neighborhoods' sizes in the enhanced power graph (as known as the cyclic graph) associated with a finite group. In particular, we characterize finite $p$-groups with the smallest maximum size for neighborhoods of nontrivial element in its enhanced power graph. 

%In this paper we aim to find structural property of a finite group $G$ studying combinatorial invariant of neighborhoods of elements in $\Gamma_{\mathcal X}(G)$, when $\mathcal X= \mathcal C$ is either the class of cyclic groups, or $\mathcal X= \mathcal N$ the class of nilpotent groups. In particular we classify all the groups  the soluble groups $G$ with nilpotent neighborhoods in $\Gamma_{\mathcal N}(G)$.

%More precisely, we characterize prove that a groups has abelian Sylow subgroups if and only if the nilpotentizer of every element of prime power order coincide with its centralizer. Next we completely describe the class of $p$-groups $G$ in which the maximum degree of vertices in the cyclic graph, say $n$, equals $\exp(G)-2$. For all the other $p$-groups a lower bound to $n$ is provided. 
\end{abstract}

\maketitle

%%%%%%%%%%%%%%%%%%%%%% SECTION

\section{Introduction}\label{I}

\noindent 
All groups considered in this paper are finite unless otherwise stated. 
To study the structure of a group, one can look at the invariants of some graphs whose vertices are the elements of the group and whose edges reveal some properties of the group itself. More precisely, if $G$ is a group and $\mathcal B$ is a class of groups, the $\mathcal B$-graph associated with $G$, denoted by $\Gamma_{\mathcal B}(G)$, is a simple and undirected graph whose vertices are the elements of $G$ and there is an edge between two elements $x$ and $y$ of $G$ if the subgroup generated by $x$ and $y$  is a $\mathcal B$-group.

Several features of a finite group can be detected analyzing the invariants of its $\mathcal B$-graph. We refer to \cite{Cameron} for a survey on this topic and to \cite{GLM} and \cite{GM} for related works.
Recent papers deal with the investigation of the (closed) neighborhood $\mathcal I_\mathcal B(x)$ of a vertex $x$ in $\Gamma_{\mathcal B}(G)$, that is, the set of all $y$ in $G$ such that $x$ and $y$ generate a $\mathcal B$-group. When $\mathcal B$ is the class of abelian groups, then $\mathcal I_\mathcal B(x)$ coincides with the centralizer of $x$ in $G$, thus $\mathcal I_\mathcal B(x)$ is a subgroup. However, in general this is not the case when $\mathcal B$ is distinct from the class of abelian groups. Nevertheless, even though $\mathcal I_\mathcal B(x)$ is not a subgroup of $G$ in general, it can happen that the characteristics of a single neighborhood in a $\mathcal B$-graph could affect the structure of the whole group $G$. For instance, when $\mathcal B$ coincides with the class $\mathcal S$ of soluble groups, it has been showed that the combinatorial properties, as well as, arithmetic ones of $\mathcal I_\mathcal B(x)$ may force the whole group to be abelian or nilpotent (see \cite{ALMM} and \cite{ADM} for more details).

Here we start considering the class $\mathcal C$ of all cyclic groups.  Cameron in \cite{Cameron} calls the graph $\Gamma_{\mathcal C} (G)$ the {\it enhanced power graph}.  However, this graph was first studied in \cite{Imp} under the name {\it cyclic graph}. Further investigations under this name occurred in \cite{ImLe}.  Recently, this graph has been investigated in \cite{CLSTU}, \cite{CLSTU1}, and \cite{CLSTU2}. 

Our interest for $\Gamma_{\mathcal C} (G)$ chiefly concerns the cardinality of $I_\mathcal C (x)$ discussing the possible values that can occur for $|I_\mathcal C (x)|$ when $x$ belong to a $p$-group $G$. Denote by $n_G$ the maximum of the sizes of all $I_\mathcal C(x)$ for $x \in G \setminus \{1\}$. Then clearly we have
\[
\exp(G) \leq n_G \leq |G|,
\]
where $\exp(G)$ denotes the exponent of the group $G$. Every time $G$ has a non trivial universal vertex, that is a nontrivial element adjacent to any element of $G$, $n_G=|G|$. These groups have been characterized in the soluble case in \cite{CLSTU2}. Our first goal is to characterize $p$-groups $G$ with $n_G = \exp(G)$. Indeed we prove the following:

\begin{thm}\label{thm:main.cyc}
Let $G$ be a finite $p$-group. Then $n_G=\exp(G)$ if and only if $G$ is either cyclic, or $\exp(G)=p$ or $G$ is a dihedral $2$-group.
\end{thm}

Going further one may ask which is the second value that can occur for $n_G$, and the answer is given by the following.

\begin{pro}\label{pro:nG}
Let $G$ be a $p$-group and assume $n_G > \exp(G)$. Then $n_G \geq p^{\alpha + 1} - p^{\alpha} + p^{\alpha-1}$.
\end{pro}

We point out that the bound in Theorem \ref{pro:nG} is sharp in some sense. Indeed, for $G=C_{p^2} \times C_p$ we have $n_G = p^{3} - p^{2} + p$, where $C_k$ denotes the cyclic group of order $k$.

%Next, we consider the class $\mathcal N$ of all {\it nilpotent groups} and, for a group $G$, the {\it nilpotent graph} $\Gamma_\mathcal N (G)$ associated with $G$. This graph was first considered in \cite{AZ}, where the authors provide some classes of $\mathfrak n$-groups.  A group $G$ is called a {\it $\mathfrak n$-group} if the subset $I_\mathcal N (x)$ is a subgroup for every element $x \in G$. It is easy to see that a group $G$ is an $\mathfrak n$-group if and only if $G/Z_{\infty}(G)$ is an $\mathfrak n$-group, where $Z_{\infty}(G)$ is the hyper center of $G$ (see Lemma 2.2 (4) of \cite{AZ}). Therefore, when looking for $\mathfrak n$-groups, we may assume that the groups have trivial center. 

%In \cite{AZ}, simple $\mathfrak n$-groups are characterized. In this paper, we begin to consider the structure of a solvable $\mathfrak n$-group. After showing some properties of the class of $\mathfrak n$-groups, we completely classify solvable groups whose neighborhoods in their nilpotent graph are nilpotent subgroups.

%\begin{thm}\label{thm:nilp}
%Let $G$ be a finite solvable group with trivial center. Then $I_{\mathcal N}(x)$ is a nilpotent subgroup of $G$ for every nontrivial $x \in G$ if and only if $G$ is a Frobenius group with nilpotent Frobenius complement.
%\end{thm}

%It is worth mentioning that the hypothesis of Theorem \ref{thm:nilp} are quite natural in the context of graph theoretic problems. Indeed, having nilpotent $I_{\mathcal N} (x)$ coincides with the requirement that every neighborhood of the graph be a clique. 

%%%%%%%%%%%%%%%%%%%%%% SECTION

\section{The cyclic graph}\label{C}

In this section we will deal with the enhanced power graph of a group or what we like to call the cyclic graph of a group. Recall that the cyclic graph of a group $G$, denoted by $\Delta (G)$, is the graph whose vertex set is $G \setminus \{1\}$ and two distinct elements $x,y$ of $G$ are adjacent if and only if $\langle x, y\rangle$ is cyclic. When $x$ and $y$ are adjacent we will write  $x \sim y$.  We denote by $n_G$ the maximum of the sizes of all $I_{\mathcal C} (x)$ for $x \in G \setminus \{1\}$.  We begin with the following useful lemma.

\begin{lem}\label{lem:connected}
Let $p$ be a prime and let $G$ be a $p$-group. Then there exists an element $z \in G$ of order $p$ such that $|I_{\mathcal C} (z)| = n_G$.
\end{lem}

\begin{proof}
Observe that there exists an element $x \in G$ such that $|I_{\mathcal C} (x)| = n_G$. If $o (x) = p$, then we are done. Therefore, we assume that $o (x) = p^k$ where $k$ is an integer so that $k \geq 2$. Take $z = x^{p^{k-1}}$, and observe that $x$ and $z$ belong to the same connected component $\Upsilon$ in $\Delta(G)$, and $z$ is the only element of order $p$ in $\Upsilon$. By Lemma 2.2 of \cite{CLSTU}, $z \sim y$ for any element $y \in \Upsilon$, and so, $|I_{\mathcal C} (z)| \geq |I_{\mathcal C} (x)|=n_G$, which implies $|I_{\mathcal C} (z)| = n_G$.  
\end{proof}

By Lemma~\ref{lem:connected} and Lemma 2.2 of \cite{CLSTU}, one can easily see that $n_G = |\Upsilon| - 1$, where $\Upsilon$ is a connected component of $\Delta (G)$ containing a vertex of degree $n_G$.  

\subsection{Abelian $p$-groups}

In this subsection, we focus on Abelian $p$--groups.  In this next lemma, we compute $n_G$ when $G$ is a non-trivial cyclic group.

\begin{lem}\label{lem:cyclic}
If $G$ is a non-trivial cyclic group, then $n_G=|G|$.
\end{lem}

\begin{proof}
Let $x \in G$ such that $G= \langle x \rangle$. Since $o(x)=|G|$  and $G \setminus \langle x \rangle = \emptyset$, we conclude that $n_G=|G|$. 
\end{proof}

We next compute $n_G$ when $G$ is a $p$-group having exponent $p$.

\begin{lem}\label{lem:elabel}
Let $p$ be a prime and let $G$ be a $p$-group of exponent $p$. Then $n_G = p$.
\end{lem}

\begin{proof}
If $G$ is a cyclic group of order $p$, then the result follows from Lemma~\ref{lem:cyclic}.  Assume that $G$ is not cyclic, and consider an element $x \in G$ such that $|I_{\mathcal C} (x)| = n_G$.  As $o(x) = p$, we have $n_G \geq p$. 

Now observe that if $y \in G \setminus \langle x \rangle$, then $ \langle x, y \rangle$ is not cyclic. Indeed, arguing by contradiction let $z \in G$ such that $\langle x, y \rangle = \langle z \rangle$. Since $G$ has exponent $p$, there exist $i , j \in \{1, \ldots, p-1\}$ such that $x = z^i$ and $y = z^j$. Therefore, from $(i,p) = 1$ it follows that $\langle x \rangle = \langle z^i \rangle = \langle z \rangle$ and $y \in  \langle x \rangle$, a contradiction.  Hence, we conclude that $n_G = p$.
\end{proof}

We now show that if $G$ is non-cyclic abelian group whose exponent is larger than $p$, then $n_G$ is larger than the exponent of $G$.

\begin{lem}\label{lem:noncyc}
Let $p$ be a prime and let $G$ be a non-cyclic abelian $p$-group of exponent $exp(G) = p^{\alpha}$, where $\alpha \geq 2$. Then $n_G \geq p^{\alpha + 1} - p^{\alpha} + p^{\alpha-1}$.  As a consequence, $n_G > \exp(G)$.
\end{lem}

\begin{proof}
As $G$ is abelian, we may assume 
$$G=C_{p^{\alpha_1}} \times \cdots \times C_{p^{\alpha_r}},$$ 
where $r\geq 2$, $1 \leq \alpha_1 \leq \cdots \leq \alpha_r=\alpha$ and $C_{p^{\alpha_i}}=\langle x_i\rangle$ is a cyclic group of order $p^{\alpha_i}$.

If $\alpha_{r-1}=1$, then the vertex $x_r^{p^{\alpha-1}}$ is adjacent to $p^{\alpha}-2$ non-trivial elements of $\langle x_r \rangle$ and to any element of the form $x_{r-1}^i x_r^k$ where $i=1, \ldots, p-1$ and $k$ is a positive integer less than $p^{\alpha}$ and coprime with $p$. Hence there are precisely $p^{\alpha} - p^{\alpha - 1}$ choices for $k$, which implies $$|I_{\mathcal C}(x)| \geq p^{\alpha} + (p-1)(p^{\alpha} - p^{\alpha-1}) = p^{\alpha + 1} - p^{\alpha} + p^{\alpha-1}.$$

If $\alpha_{r-1} >1$, then one can consider the subgroup $\langle x_r^{p^{\alpha_{r-1} - 1}}, x_r \rangle$ arguing as in the previous case.
\end{proof}

We now collect these lemmas in a proposition where we note that for an abelian $p$-group $G$ that $n_G$ equals the exponent of $G$ if and only if $G$ is cyclic or elementary abelian.

\begin{pro}\label{prop:abelian}
Let $p$ be a prime and let $G$ be an abelian $p$-group. Then $n_G=\exp(G)$ if and only if $G$ is either cyclic or elementary abelian.
\end{pro}

\begin{proof}
If $G$ is either cyclic or elementary abelian, then the result follows from Lemma~\ref{lem:cyclic} and Lemma~\ref{lem:elabel}. Conversely, assume that $n_G = \exp(G)$. If $G$ is neither cyclic nor elementary abelian, then applying Lemma~\ref{lem:noncyc} we have $n_G > \exp(G)$, a contradiction.
\end{proof}

\subsection{Nonabelian $p$-groups}

We now shift our focus to nonabelian $p$-groups.  When $p$ is a prime, we take $\alpha$ to be an integer greater than $1$ when $p$ is odd and an integer greater than $2$ when $p = 2$. We denote by $M_{p^{\alpha +1}}$ the group
$$
M_{p^{\alpha +1}} = \langle x,y \mid x^{p^{\alpha}}=y^p=1, \ x^y=x^{p^{\alpha-1}+1} \rangle.
$$
Going further, we denote by $D_{2^{\alpha +1}}$, $S_{p^{\alpha +1}}$ and $Q_{2^{\alpha +1}}$ the dihedral, semidihedral, and generalized quaternion groups given by the following presentations:
$$
D_{2^{\alpha +1}} = \langle x,y \mid x^{2^{\alpha}}=y^2=1, \ x^y=x^{-1} \rangle,
$$
$$
S_{p^{\alpha +1}} = \langle x,y \mid x^{p^{\alpha}}=y^p=1, \ x^y=x^{p^{\alpha-1}-1} \rangle,
$$
$$
Q_{2^{\alpha +1}} = \langle x,y \mid x^{2^{\alpha}-1}=y^2, \ y^4=1,\ x^y=x^{-1} \rangle.
$$

The characterization of non-abelian $p$-groups with a cyclic maximal subgroup is well-known (see \cite{GOR}).

\begin{thm}\label{thm:class}
Let $p$ be a prime and let $G$ be a non-abelian $p$-group of order $p^{\alpha+1}$ with a cyclic subgroup of order $p^{\alpha}$. Then
\begin{enumerate}
\item[(i)] if $p$ is odd then $G$ is isomorphic to $M_{p^{\alpha + 1}}$;
\item[(ii)] if $p=2$ and $\alpha=2$, then $G$ is isomorphic to either $D_{8}$ or $Q_{8}$;
\item[(iii)] if $p=2$ and $\alpha>3$ then $G$ is isomorphic to either $M_{2^{\alpha + 1}}$, $D_{2^{\alpha + 1}}$, $Q_{2^{\alpha + 1}}$ or $S_{2^{\alpha + 1}}$.
%$\langle x,y \mid x^{2^{\alpha}}=y^2=1, \ x^y=x^{2^{\alpha-1}+1} \rangle$, or $\langle x,y \mid x^{2^{\alpha}}=y^2=1, \ x^y=x^{2^{\alpha-1}-1} \rangle$.
\end{enumerate}
\end{thm} 

We compute $n_G$ for nonabelian $p$-groups with a maximal cyclic subgroup of index $p$.

\begin{pro}\label{pro:maxcyc}
Let $p$ be a prime and let $G$ be a $p$-group of order $p^{\alpha+1}$. Assume that $G$ has a maximal cyclic subgroup of order $p^{\alpha}$. Then $n_G = \exp (G)$ if and only if either $G$ is cyclic, or $\exp (G) = p$, or $G \simeq D_{2^{\alpha+1}}$.
\end{pro}

\begin{proof}
If $G$ is cyclic or $\exp(G) = p$, then $n_G = \exp (G)$ by Lemmas \ref{lem:elabel} and \ref{lem:cyclic}.  Moreover, if $G \simeq D_{2^{\alpha+1}}$, then $G$ has only one cyclic subgroup of order $2^{\alpha}$ while all the other cyclic subgroups have order $2$, which implies $n_G = \exp(G)$.

Now, assume that $n_G = \exp(G)$. If $G$ is abelian then $G$ is either cyclic or elementary abelian by Proposition~\ref{prop:abelian}. Now assume that $G$ is neither abelian nor of exponent $p$. From Theorem~\ref{thm:class} we have to analyze two cases. First assume that $G$ is isomorphic to $M_{p^{\alpha + 1}}$. %$\langle x,y \mid x^{p^{\alpha}}=y^p=1, \ x^y=x^{p^{\alpha-1}+1} \rangle$.
Then $(yx)^p = x^{\frac{p(p-1)}{2}p^{\alpha-1}+p}$ which yields a contradiction.  Indeed, when $p$ is odd $(yx)^p = x^p$ and $|I_{\mathcal C} (x^p)| > \exp(G)$ as $x^p$ is connected to every element of $\langle x \rangle$ and to every element of $\langle yx \rangle$.  If $p=2$, then $(yx)^2 = x^{2^{\alpha - 1} +2}$, and $I_{\mathcal C}(x^{2^{\alpha-1} + 2})$ contains more than $2^{\alpha}$ elements. 

Finally, assume that $p = 2$ and $G$ isomorphic to $S_{2^{\alpha + 1}}$. %$\langle x,y \mid x^{2^{\alpha}}=y^2=1, \ x^y=x^{2^{\alpha-1}-1} \rangle$. 
Then $(yx)^2 = x^{2^{\alpha-1}}$ and $|I_{\mathcal C} (yx)| > \exp (G)$.
\end{proof}

We are now in a position to prove Theorem \ref{thm:main.cyc}.

\begin{proof}[Proof of Theorem \ref{thm:main.cyc}]

By Lemmas \ref{lem:cyclic} and \ref{lem:elabel} and Proposition \ref{pro:maxcyc}, we only need to prove that if $n_G=\exp(G)$ then $G$ is either cyclic, or $\exp(G)=p$ or $G$ is a dihedral $2$-group. Thus let $n_G=\exp(G)$, and %If $G$ is abelian, then Proposition \ref{prop:abelian} implies that either $G$ is cyclic or $\exp(G)=p$. Therefore we can assume that $G$ is not abelian. If $x \in G \setminus Z(G)$.
by way of contradiction assume that $G$ is neither cyclic, nor $\exp (G) = p$, nor a dihedral group of order $2^{\exp (G) + 1}$, such that $G$ has minimal order. Hence, there exists an element $x \in G$ such that $p < o(x) = \exp (G)$. By Proposition \ref{pro:maxcyc}, it follows that $p \cdot o(x) < |G|$, and thus, $G$ contains a proper subgroup $H$ such that $x \in H$ and  $|H|=p\cdot o(x)$. Then $\exp (H) = \exp (G)$, and $H$ has a cyclic subgroup of index $p$.  By Proposition \ref{pro:maxcyc}, $H$ is a dihedral group of order $2 \exp(G)$ since $H$ is neither cyclic nor $\exp(H) = p$.  As a consequence $G$ is a $2$-group, and by minimality, $|G:H|=2$. If $o(x)=4$, then $|G|=16$ and an easy computation using GAP shows that this is a contradiction. Hence we may assume $o(x)>4$. Now assume that there exists an element $a \in G \setminus H$ such that $o(a)>4$. Then $a^2 \in H$ and $o(a^2)>2$. This implies that $a^2 \in \langle x \rangle$ and $|I_{\mathcal C} (a^2)| > \exp (G)$. Hence we may assume that $o(a) \leq 4$ for all $a \in G \setminus H$. First assume that $G \setminus H$ contains an element $a$ of order $2$. If $a$ does not invert $x$, then $(xa)^2=xx^a$ is a nontrivial element of $\langle x \rangle$, since $\langle x \rangle$ is normal in $G$. As a consequence, $|I_{\mathcal C} ((xa)^2)| > \exp (G)$. Now assume that $x^a=x^{-1}$. Let $b \in H$ such that $x^b=x^{-1}$. Then $x^{ab}=x$ and $ab$ belongs to the centralizer in $G$ of $x$. Thus, $(xab)^4=x^4 \neq 1$, and $|I_{\mathcal C} (x^4)| > \exp (G)$. Therefore we only need to address the case in which $o(a)=4$ for every $a \in G \setminus H$. If $a^2 \in \langle x \rangle$ for some $a \in G \setminus H$, then $|I_{\mathcal C}(a^2)|> \exp(G)$. This implies that $a^2 \in H \setminus \langle x \rangle$. As a consequence $a^2$ inverts $x$. On the other hand, the dihedral groups have no automorphisms of order $4$ whose square inverts its element of maximal order (see for instance Theorem 34.8 (a) of \cite{pgrps}). The final contradiction proves the theorem.
\end{proof}

\section*{Acknowledgements}
\noindent This work was partially supported by the National Group for Algebraic and Geometric Structures, and their Applications  (GNSAGA -- INdAM). This work was carried out during the second author's visit to the Kent State University. He wishes to thank the Department of Mathematical Science for the excellent hospitality.

\bibliographystyle{plain}

\end{document}